\documentclass[12pt,leqno]{amsart}
\usepackage{graphicx}
\usepackage[utf8]{inputenc}
\usepackage[
backend=biber,
style=alphabetic,
sorting=nyt
]{biblatex}
\usepackage[T1]{fontenc}
\usepackage{lmodern}
\usepackage{geometry}
\usepackage{booktabs}
\usepackage{xcolor}
\usepackage{amsthm}
\usepackage{amsmath,amssymb}
\usepackage{listings}
\usepackage{emptypage}
\usepackage{newlfont}
\usepackage{enumerate}
\usepackage{mathrsfs}
\usepackage{tikz-cd}
\usepackage{soul}
\newcommand{\numberset}{\mathbb}
\newcommand{\N}{\numberset{N}}
\newcommand{\Z}{\numberset{Z}}
\newcommand{\R}{\numberset{R}}
\newcommand{\C}{\numberset{C}}
\renewcommand{\H}{\numberset{H}}
\renewcommand{\P}{\numberset{P}}
\renewcommand{\O}{\numberset{O}}

\renewcommand{\S}{\numberset{S}}
\newcommand{\RP}{\mathbb{RP}}
\newcommand{\CP}{\mathbb{CP}}
\newcommand{\HP}{\mathbb{HP}}

\newcommand{\CPn}{\mathbb{CP}^{n}}
\newcommand{\HPn}{\mathbb{HP}^{n}}
\newcommand{\OPP}{\O\mathbb{P}^{2}}
\newcommand{\OP}{\O\mathbb{P}}

\newcommand{\msuF}{M/\mathcal{F}}
\newcommand{\F}{\mathcal{F}}
\renewcommand{\L}{\mathcal{L}}

\newcommand{\fol}{\mathcal{F}}
\newcommand{\hili}[1]{#1}

\makeatletter
\newcommand\footnoteref[1]{\protected@xdef\@thefnmark{\ref{#1}}\@footnotemark}
\makeatother

\theoremstyle{plain} 
\newtheorem{thm}{Theorem}[section]
 
\newtheorem{lem}[thm]{Lemma}

\newtheorem*{thm*}{Theorem}
\newtheorem*{cor*}{Corollary} 
\newtheorem*{lem*}{Lemma} 
\newtheorem*{prop*}{Proposition}
\newtheorem*{rmk*}{Remark}
\newtheorem{mainthm}{Theorem}

\theoremstyle{definition}

\newtheorem*{defn*}{Definition}
\newtheorem*{ex*}{Example}

\theoremstyle{remark}

\begin{document}

\pagebreak

\title{Riemannian Foliations On Homotopy CROSSes}

\author{Marco Radeschi}
\address{Universit\`a degli Studi di Torino\newline
\indent Departimento di Matematica ``G. Peano''\newline
\indent Via Carlo Alberto, 10\newline
\indent 10123 Torino (TO), Italy}
\email{marco.radeschi@unito.it}

\author{Lorenzo Scoffone}
\address{KIT\newline
\indent Kollegiengebäude Mathematik\newline
\indent Englerstraße 2\newline
\indent 76131 Karlsruhe, Germany}
\email{lorenzo.scoffone@kit.edu}

\author{Michael Wiemeler}
\address{Universität Münster\newline
\indent Mathematisches Institut\newline
\indent Einsteinstraße 62\newline
\indent 48149 Münster, Germany}
\email{wiemelerm@uni-muenster.de}

\begin{abstract}
We classify Riemannian foliations of manifolds homotopy equivalent to CROSSes.
\end{abstract}
\maketitle

\section{Introduction}

The problem of classifying Riemannian foliations on round spheres took a surprisingly long time to be fully solved, spanning several work of several people \cite{Ran85}, \cite{gromollgrove}, \cite{Wil01}, \cite{lytchakwilking}. In \cite{lytchakwilking} the problem was finally solved, in fact for Riemannian foliations on homotopy spheres:

\begin{thm}[\cite{lytchakwilking}]
    \label{thmfoliationsonspheres}
    Consider a $k$-dimensional Riemannian foliation $\mathcal{F}$ of a Riemannian manifold $(M,g)$ homotopic to $\S^{n}$. Assuming $0<k<n$, one of the following holds:
    \begin{enumerate}[(i)]
        \item $n=2l+1$ for some $l\in\N_{>0}$, $k=1$ and the foliation is given by an isometric flow, up to changing the Riemannian metric.
        \item $n=4l+3$ for some $l\in\N_{>0}$, $k=3$ and the generic leaves are diffeomorphic to $\S^{3}$ or $\RP^{3}$.
        \item $n=15$, $k=7$ and $\mathcal{F}$ is simple, given by the fibres of a Riemannian submersion $(M,g)\to(B,g_{B})$ with $(B,g_{B})$ homeomorphic to $\S^{8}$ and with fibres homeomorphic to $\S^{7}$.
    \end{enumerate}
    Furthermore all these cases can occur.
\end{thm}

\hili{In this paper, we complete the classification of Riemannian foliations on manifolds homotopic to the remaining simply connected CROSSes, namely $\CP^n$, $\HP^n$ and $\OP^2$. Recall that, for any $m\geq 1$ the \emph{twistor bundle} $T:\CP^{2m+1}\to\HP^{m}$ is the Riemannian submersion with $\S^2$-fibers, given by}
$$\left[x_{0}:\ldots:x_{2m+1}\right]\mapsto\left[x_{0}+x_{1}j:\ldots: x_{2m}+x_{2m+1}j\right].$$

\begin{mainthm}
    \label{thmfoliationsonCROSSes}
    Consider a $k$-dimensional Riemannian foliation $(M,\mathcal{F})$, $0<k<\dim M$, of a Riemannian manifold $(M,g)$ \hili{homotopy equivalent to} a simply connected non-spherical CROSS. 
    Then:
    \begin{enumerate}
    \item $M$ is \hili{homotopy equivalent} to $\CP^{2m+1}$ for some $m\in\N_{>0}$ and the foliation is given by the fibers of an $\S^2$-bundle $\pi:(M,g)\to(B,g_{B})$, with $B$ \hili{homotopy equivalent to $\HP^m$. Furthermore, $\pi$ is obtained by pulling back the twistor bundle $T:\CP^{2m+1}\to\HP^{m}$ via a homotopy equivalence $B\to \HP^m$.}
    \item If $M$ is isometric to $\CP^{2m+1}$ with its canonical metric, then the Riemannian submersion $\CP^{2m+1}\to B$ is congruent to the twistor bundle.
    
    \end{enumerate}
\end{mainthm}
In particular, no non-trivial Riemannian foliation can occur on manifolds \hili{homotopy equivalent} to $\HP^{n}$ or $\OPP$, and any homotopy CROSS with a nontrivial Riemannian foliation has dimension $4m+2$ for some $m\geq 1$.

\hili{The theorem is sharp, in the sense that $M$ is not necessarily homeomorphic to $\CP^{2m+1}$} in dimensions $\geq 10$:

\begin{mainthm}\label{MT:exotic}
If $M$ is a $6$-dimensional homotopy CROSS equipped with a Riemannian foliation, then it is diffeomorphic to $\CP^{3}$. Conversely, for any $m\geq 2$ there are infinitely many $(4m+2)$-dimensional Riemannian manifolds $M_i$ homotopy equivalent but not homeomorphic to $\CP^{2m+1}$, endowed with a Riemannian foliation.
\end{mainthm}

\subsection*{Acknowledgments}

The authors would like to thank Alexander Lytchak for suggesting the problem and for insightful discussions.

Michael Wiemeler was funded by the Deutsche Forschungsgemeinschaft (DFG, German Research Foundation) under Germany’s Excellence Strategy EXC 2044/2 –390685587, Mathematics Münster: Dynamics – Geometry – Structure and through CRC 1442, Geometry: Deformations and Rigidity at University of Münster.

\section{$M$ homotopic to $\CPn$}
\label{secCPn}
Fix some Riemannian manifold $(M,g)$ \hili{homotopy equivalent to} $\CPn$, and fix a \hili{homotopy equivalence} $f:M\to \CPn$. \hili{Since continuous maps can be approximated by smooth ones, we can assume that $f$ is a smooth map}. The Hopf fibration $H_n:\S^{2n+1}\to \CPn$ pulls back to a ``Hopf-like'' principal $\S^{1}$-bundle $\bar{H}_{n}:f^{*}(\S^{2n+1})\to M$:

\begin{center}
\begin{tikzcd}
f^*(\S^{2n+1})\arrow[r,"F"]\arrow[d,"\bar{H}_n"]& \S^{2n+1}\arrow[d,"{H}_n"]\\
M\arrow[r,"f"]&\CPn
\end{tikzcd}
\end{center}

\hili{$f^*(\S^{2n+1})$ is a smooth manifold, and $\bar{H}_n$ is a (smooth) principal bundle. Furthermore, by the long exact sequence in homotopy and the 5-lemma, it follows that $F$ induces an isomorphism between the homotopy groups of $f^*(\S^{2n+1})$ and $\S^{2n+1}$, and in particular $f^*(\S^{2n+1})$ is homeomorphic to $\S^{2n+1}$.}
We endow $f^*(\S^{2n+1})$ with a metric which makes the $\S^1$-action isometric, and the map $\bar{H}_n$ a Riemannian submersion.

Assume there exists a Riemannian foliation $(M,\mathcal{F})$ with $k$-dimensional fibres, with $0<k<2n$. The idea is to pullback $\mathcal{F}$ to a foliation $\bar{\fol}:=\bar{H}_{n}^{-1}(\fol)$ on $f^{*}(\S^{2n+1})$, and then use the classification given in \cite{lytchakwilking}.\\

Denote by $\mathcal{L}_{p}$ the leaf of $\mathcal{F}$ through $p\in M$ and by $\bar{\mathcal{L}}_{q}$ the corresponding leaf in $\bar{\mathcal{F}}$ through $q\in \bar{H}_{n}^{-1}(p)$. Then,
\[
\dim(\bar{\mathcal{L}_{q}})=\dim(\mathcal{L}_{p})+\dim((\bar{H}_{n})^{-1}(p))=\dim(\mathcal{L}_{p})+\dim(\S^{1})=k+1>1\text{.}
\]
By Theorem \ref{thmfoliationsonspheres}, $\dim(\bar{\mathcal{L}}_{q})=3$ or $7$, meaning $\dim(\mathcal{L}_{p})=2$ or $6$. We consider these cases separately, although with analogous methods.\\

From now on, we denote by $\mathcal{L}$ any leaf of $\fol$ and let $\bar{\mathcal{L}}=(\bar{H}_{n})^{-1}(\mathcal{L})\in\bar{\fol}$.

\subsection{Case 1: $\bar{\mathcal{F}}$ has $3$-dimensional fibres.}
We know by \cite{lytchakwilking} that, in this case, the generic leaf of $\bar{\fol}$ is diffeomorphic to $\S^{3}$ or $\RP^{3}$. This means that any leaf $\bar{\mathcal{L}}$ is covered by $\S^3$ and in particular $\pi_1(\bar{\mathcal{L}})$ is a finite group. Furthermore, by \cite{lytchakwilking} a $3$-dimensional foliation can occur only on a sphere of dimension $4m+3$ for some $m\in\N_{>0}$. Thus $n=2m+1$.\\

Note that the Hopf-like principal $\S^1$-bundle $\bar{H}_n:f^{*}(\S^{4m+3})\to M$ restricts to a principal $\S^{1}$-bundle $\S^{1}\to\bar{\mathcal{L}}\to\mathcal{L}$. If $\bar{\mathcal{L}}$ is diffeomorphic to $\S^{3}$, $\mathcal{L}$ is simply connected by the long exact sequence in homotopy of a fibration, and by the classification of compact and connected surfaces $\mathcal{L}$ is then diffeomorphic to $\S^{2}$.\\

To get the same result when $\bar{\mathcal{L}}$ is non-trivially covered by $\S^3$, we consider again the long exact sequence in homotopy of this fibration:
\[
    0=\pi_{2}(\bar{\mathcal{L}})\to\pi_{2}(\mathcal{L})\to\pi_{1}(\S^{1})\stackrel{\phi}{\to}\pi_1(\bar{\mathcal{L}})\to\pi_{1}(\mathcal{L})\to0\text{.}
\]
By the first part of this sequence, $\pi_{2}(\mathcal{L})\cong\ker(\phi)\neq0
$ (since $\pi_1(\bar{\mathcal{L}})$ is finite): the only two-dimensional manifolds for which this is true are $\S^{2}$ or $\RP^{2}$. Let us now try to exclude the $\RP^{2}$ case. The fibration $\S^{1}\to\bar{\mathcal{L}}\to\mathcal{L}$ is orientable, so that we can consider its Gysin sequence:
\[
    \ldots\to H^{1}(\bar{\mathcal{L}};\Z)\to H^{0}(\mathcal{L};\Z)\to H^{2}(\mathcal{L};\Z)\to\ldots
\]
Now, $H^{1}(\bar{\mathcal{L}};\Z)=\mathrm{Hom}(\pi_1(\bar{\mathcal{L}}),\Z)=0$ since $\Z$ has no torsion. Therefore, $H^{0}(\mathcal{L};\Z)=\Z$ has an injective image into $H^{2}(\mathcal{L};\Z)$. On the other hand, $H^2(\mathcal{L};\Z)=\Z$ or $\Z_{2}$ if $\mathcal{L}\cong\S^{2}$ or $\RP^{2}$, respectively. Then necessarily $H^{2}(\mathcal{L};\Z)=\Z$ and $\mathcal{L}$ is homeomorphic to $\S^{2}$ and, in particular, simply connected.

Since all leaves are simply connected, by Theorem 2.2 of \cite{Escobales} the leaf space $B:=M/\mathcal{F}$ is a Riemannian manifold and $M\to B,\ p\mapsto\mathcal{L}_{p}$ is a Riemannian submersion\footnote{Since the leaves are simply connected, the leaf holonomy groups used in \cite{Escobales} are trivial: see \cite{Hermann}, p.448.}: $\mathcal{F}$ is \emph{simple}.

Note that, given a Riemannian submersion $\pi:N\to M$ and a Riemannian foliation $(M,\fol)$ whose leaves are given by fibres of some Riemannian submersion $\pi':M\to B$, then the leaves of the pullback foliation $(N,\pi^{-1}(\fol))$ are given by the fibres of $\pi'\circ\pi$, which is again a Riemannian submersion by composition. Therefore, the pullback of a simple Riemannian foliation is again simple, and $N/\pi^{-1}(\fol)=B=M/\fol$.\\

Thus in our case $(f^{*}(\S^{4m+3}),\bar{\fol})$ is simple as well, given by the fibres of a Riemannian submersion $f^{*}(\S^{4m+3})\to B$. Since $f^*(\S^{4m+3})$ is a topological sphere, Theorem 5.1 of \cite{browder} implies that $\bar\L\simeq \S^3$ and (by the discussion in the introduction of \cite{browder}, or simply by considering the Gysin sequence of the fibration $\S^3\to f^*(\S^{4m+3})\to B$), $H^*(B;R)\simeq H^*(\HP^m;R)$ for any ring $R$.

\hili{The free $\S^1$-action on $f^*(\S^{4m+3})$ restricts to the $\S^3$ fibers of $f^*(\S^{4m+3})\to B$, thus the fibers of $M=f^*(\S^{4m+3})/\S^1\to B$ are diffeomorphic to $\S^2$. The fact that $B$ is homotopy equivalent to $\HP^m$ and $M\to B$ is the pullback of the twistor fibration then follows from the following Lemma:}
\begin{lem}
\hili{Let $\S^2\to M \to B$ be an $\S^2$-bundle over a simply connected base $B$ with $H^*(B;\Z)=H^*(\HP^m;\Z)$ such that $\dim M=4m+2$ and $M$ has the cohomology of
$\CP^{2m+1}$, $m\geq 1$. Then $B$ is homotopy equivalent to $\HP^m$, and $M\to B$ is the pull-back of the twistor fibration $\CP^{2m+1}\to \HP^m$ via the homotopy equivalence $B\to \HP^m$.}
\end{lem}
\begin{proof}
By Smale's theorem \cite{Sma}, the
structure group of the bundle $M\to B$ reduces from $\operatorname{Diff}(\S^2)$ to $\operatorname{SO}(3)$. Furthermore, since $B$ is $2$-connected by Hurewicz, the structure group lifts to $\operatorname{Spin}(3)=\operatorname{SU}(2)$. Thinking of $\S^2$ as $\P(V_0)=(V_0\setminus \{0\})/\C^*$ with $V_0=\C^2$, it follows that the $\S^2$-bundle $M$ is the projectivization of some $\C^2$-vector bundle $V\to B$. From e.g. \cite[Equation (20.6)]{BT82} one has an isomorphism
\[
H^*(M;\Z)=H^*(B;\Z)[x]/(x^2+c_2(V))
\]

where $x\in H^2(M;\Z)$ is the first Chern class of the tautological bundle
over $M=\P(V)$ and $c_2(V)$ is the second Chern class of $V$.

Since $x\in H^2(M;\Z)$ is a generator and $H^*(M;\Z)\simeq H^*(\CP^{2m+1};\Z)$, in particular $x^2$ generates $H^4(M;\Z)$ and thus $c_2(V)$ generates $H^4(B;\Z)$.

Hence the classifying map for $V$, $\varphi:B\to B\operatorname{SU}(2)=\HP^{\infty}$, whose induced map
\[
\varphi^*:H^*(\HP^\infty;\Z)\simeq \Z[x]\to H^*(B;\Z)\simeq \Z[x]/(x^{m+1})
\]
sends $x\in H^4(\HP^\infty;\Z)$ to $c_2(V)$ generator of $H^4(B;\Z)$, in particular induces an isomorphism between the integral cohomology ring of $B$ and that of the $(4m+1)$-skeleton $X= \HP^m$ of $\HP^\infty$. Therefore, $\varphi:B\to X$ induces a homotopy equivalence. Finally, $M\to B$ is the pull-back of the universal bundle
\[
\S^2\to \S^2\times_{\operatorname{SU}(2)}E\operatorname{SU}(2)\to B\operatorname{SU}(2)
\]
This bundle is homotopy equivalent to
\[
\operatorname{SU}(2)/\S^1\to B\S^1\to B\operatorname{SU}(2)
\]

via the homeomorphism
\[
\S^2\times_{\operatorname{SU}(2)}E\operatorname{SU}(2)\to B\S^1=E\operatorname{SU}(2)/\S^1,\qquad [gv_0,x] \mapsto [xg^{-1}]
\]
with $v_0\in \S^2$ fixed a priori. A model of $B\S^1$ (resp. $B\operatorname{SU}(2)$) is the direct limit of $\CP^n$'s (resp. of $\HP^n$), and the map $B\S^1\to B\operatorname{SU}(2)$ is the direct limit of the twistor fibrations
\[
\S^2\to \CP^{2m+1}\to \HP^m.
\]
In particular, the restriction of this bundle to $X\simeq \HP^m$ is the twistor fibration. Since the classifying map $\varphi$ factors through $B\to X\to \HP^\infty$, it follows that $M\to B$ is the pull back of the twistor fibration.
\end{proof}

The case of $M$ isometric to $\CP^{2m+1}$ with its canonical metric will be discussed later in Section \ref{secCPncanonical}.

\subsection{Case 2: $\bar{\mathcal{F}}$ has $7$-dimensional fibres.}\label{SS:case2}
We know that in this case the only possibility is $2n+1=15$, so that $n=7$. Proceeding like before, we get:\\

\begin{center}
\begin{tikzcd}
f^*(\S^{15})\arrow[r,"F"]\arrow[d,"\bar{H}_7"]& \S^{15}\arrow[d,"{H}_7"]\\
M\arrow[r,"f"]&\CP^7
\end{tikzcd}
\end{center}

In this case, $\bar{\mathcal{L}}$ is homeomorphic to $\S^{7}$ and just like before $\S^{1}\to\S^{7}\to\mathcal{L}$ is a fibration and $\dim(\mathcal{L})=6$. Since $\S^{1}$ is connected and $\pi_{1}(\S^{7})=0$, we get $\pi_{1}(\mathcal{L})=0$. As before, Theorem 2.2 of \cite{Escobales} implies that $B:=\msuF$ is a manifold and $\mathcal{F}$ is simple, and from Theorem \ref{thmfoliationsonspheres} we conclude that $B$ is homeomorphic to $\S^{8}$. \hili{ Furthermore, the free $\S^1$-action on $f^*(\S^{15})$ restricts to the $\S^7$ fibers of $\bar{H}_7$, hence the fibers of $M\to B$ are quotients $\S^7/\S^1$, hence homotopy equivalent to $\CP^3$.}
We will prove that such a foliation cannot exist on $M$.\\

Write $H^*(M;\Z)\simeq \Z[c]/(c^8)$. The Riemannian submersion $\pi:M\to B$ (whose fibres are $6$-dimensional) induces a splitting $TM\cong\mathcal{H}\oplus\mathcal{V}$, where $\mathcal{V}=\ker d\pi$ is the \emph{vertical distribution}, and $\mathcal{H}=\mathcal{V}^\perp$ is the \emph{horizontal distribution}. In particular, the total Pontryagin class $p(TM)\in H^*(M;\Z)$ decomposes as a product $p(TM)=p(\mathcal{V})\cup p(\mathcal{H})$. 
\\

\hili{The inclusion of a fiber $\CP^3\simeq L\stackrel{i}{\to} M$ induces a map}
\[
i^*:H^4(M;\Z)\to H^4(L;\Z).
\]
\hili{Since $B$ is homotopic to $\S^8$, by the spectral sequence of $M\to B$ the map $H^4(M;\Z)\to H^4(L;\Z)$ is an isomorphism. Since $i^*(\mathcal{V})=TL$, it follows that}$$i^*p_1(TM)=p_1(TL).$$

\hili{On the one hand, since $M$ is homotopy equivalent to $\CP^7$, it follows by}
\cite[proof of Theorem 5.1]{Des02}
\hili{that $p_1(TM)\equiv p_1(T\CP^7)=8c^2\mod 24$. 

On the other hand, since $L$ is a homotopy $\CP^3$, again by} \cite[proof of Theorem 5.1]{Des02} \hili{it follows that}
\[
p_1(TL)\equiv p_1(T\CP^3)=4c^2\quad \mod 24
\]
\hili{providing a contradiction.}

\section{$\CP^{2n+1}$ with its canonical metric}
\label{secCPncanonical}

Let us now consider the case of $M$ isometric to $\CP^{2m+1}$ with its canonical metric. In this section, we prove that any Riemannian foliation is metrically congruent to the foliation by the fibers of the twistor fibration $T:\CP^{2m+1}\to \HP^m$
\[
T[x_0:\ldots:x_{2n+1}]=[x_0+x_1j:\ldots :x_{2n}+x_{2n+1}j]
\]
To avoid confusion, in this section we will use the following notation:
\begin{itemize}
    \item $\H$ is spanned by quaternionic units $1,i,j,k$ and can be identified with $\R^4$ via $a+bi+cj+dk\mapsto (a,b,c,d)$.
    \item $Sp(1)\subset \H$ denotes the subset of quaternions of unit norm. Elements in $Sp(1)$ will be denoted with letters $p,q,r$.
    \item Given $p\in Sp(1)$ and $v\in \H^n$, the left multiplication of $v$ by $p$ is denoted $pv$. Similarly, given a matrix $A\in O(4n)$ and a vector $v\in \H^n\simeq \R^{4n}$, the multiplication of $A$ and $v$ is denoted $Av$.
    \item Any other action of $Sp(1)$ on $\H^n\simeq \R^{4n}$ is denoted by $p\cdot v$.
\end{itemize}
Proceeding like before by taking a Riemannian foliation $(\CP^{2m+1},\mathcal{F})$ and pulling back through the (standard) complex Hopf fibration $H_{2m+1}:\S^{4m+3}\to\CP^{2m+1}$, we obtain a Riemannian foliation $(\S^{4m+3},\bar{\mathcal{F}})$ of the sphere with its canonical metric.

By \cite{gromollgrove}, $\bar{\mathcal{F}}$ is homogeneous, i.e. given by the orbits of a representation $\tau:Sp(1)\to O(4m+4)$ inducing an action $p\cdot v:=\tau(p)v$ for any $p\in Sp(1)$ and $v\in \R^{4m+4}$. On the other hand, we proved in the previous section that $\bar{\mathcal{F}}$ is a simple foliation given by the fibres of the Riemannian submersion $\pi:\S^{4m+3}\to B$. This implies that the action induced by $\tau$ is free on $\S^{4m+3}$. By the classification of such actions (cf. \cite[Corollary 5.4]{gromollgrove}), by identifying $\R^{4m+4}$ with $\H^{m+1}$ in the usual way, the representation $\tau$ is equivalent to the diagonal embedding $Sp(1)\to Sp(m+1)\to O(4m+4)$. In other words, there exists an isometry $\phi: \H^{m+1}\to \H^{m+1}$ such that
\[
p\cdot v=\phi (p\phi^{-1}(v)),\qquad \forall p\in Sp(1),\quad v\in \H^{m+1}.
\]

Recall furthermore that, in our case, $\bar{\mathcal{F}}$ is the pullback of $\mathcal{F}$ via the \emph{standard} Hopf fibration $\S^{4m+3}\to \CP^{2m+1}$. In particular, the $Sp(1)$-orbits of $\tau$ contain the orbits of the standard complex Hopf fibration. Under these assumptions, we prove:

\begin{lem}
There exists a homomorphism $\rho:U(1)\to Sp(1)$, $\rho(z)=rzr^{-1}$ for some $r\in Sp(1)$, such that $\rho(e^{it})\cdot v=e^{it}v$.
\end{lem}
\begin{proof}
For any $v\in \H^{n+1}$, let $\ell_v=\operatorname{span}(v,iv, jv, kv)\simeq \H$ denote the standard quaternionic line through $v$. Since the $Sp(1)$-action is free and isometric, for any $v\in \H^{n+1}$ the map $Sp(1)\to \H^{n+1}$, $p\mapsto p\cdot v$ defines a bijection between $Sp(1)$ and the unit sphere of $\phi(\ell_{\phi^{-1}(v)})$. Since the $Sp(1)$-action contains the complex Hopf fibers in its orbits, fixing a vector $v_0$, there is a (unique) $q\in Sp(1)$ such that $q\cdot v_0=iv_0$. 

We claim that $q\cdot v=iv$ for any $v\in \H^{n+1}$. Given $v'\in \H^{n+1}$ there exist unique $q'$ such that $q'\cdot v'=iv'$. Similarly, for $v_0+v'$ there exits a unique $q''$ such that $q''\cdot (v_0+v')=i(v_0+v')$. Therefore:
\[
    q''\cdot v_0+q''\cdot v'=q''\cdot (v_0+v')=iv_0+iv'=q\cdot v_0+q'\cdot v'
\]
\[
    \Rightarrow\quad (q''-q)\cdot v_0+(q''-q')\cdot v'=0.
\]
If $\phi^{-1}(v_0)$, $\phi^{-1}(v')$ belong to different quaternionic lines, then $\H\cdot v_0\cap \H\cdot v'=\{0\}$ and thus $q=q''=q'$, thus proving the claim in this case. If $\phi^{-1}(v_0)$, $\phi^{-1}(v')$ belong to the same quaternionic line, then by taking $v''$ such that $\phi^{-1}(v'')$ belongs to a different quaternionic line, we would get $q=q''=q'$, thus proving the claim anyway.

Notice that $(q^2)\cdot v_0=(i)^2v_0=-v_0=(-1)\cdot v_0$ and again by the bijectivity of $Sp(1)\to \H^{n+1}$, it follows that $q^2=-1$. This means that $q=ri\bar{r}$ for some $r\in Sp(1)$. Finally, let $\rho:U(1)\to Sp(1)$ be given by $\rho(z)=rzr^{-1}$. Then for any $v\in \H^{n+1}$,
\begin{align*}
\rho(e^{it})\cdot v&=(r(\cos t+i\sin t)r^{-1})\cdot v=(\cos(t)+\sin(t)q)\cdot v\\
&=\cos(t)v+ \sin(t)iv=e^{it}v.
\end{align*}
\end{proof}

Define now $\psi:\R^{4n+4}\to \R^{4n+4}$ by $\psi(v)=r^{-1}\phi(v)$. Then:
\begin{enumerate}
\item The equation of the lemma gives:
\begin{align*}
\phi(e^{it}v)=&\phi(\rho(e^{it})\cdot v)= \rho(e^{it})\phi(v)=re^{it}r^{-1}\phi(v)\\
\Rightarrow \psi(e^{it}v)=&r^{-1}\phi(e^{it}v)=e^{it}\psi(v).
\end{align*}

Thus $\psi:\S^{4n+3}\to \S^{4n+3}$ induces a map $\bar{\psi}:\CP^{2n+1}\to\CP^{2n+1}$. 
\item $\psi$ sends orbits of the $Sp(1)$-action induced by $\tau$ to orbits of the standard $Sp(1)$-action: in fact, for a fixed $v\in \H^{n+1}$
\begin{align*}
\psi\left(\{g\cdot v\mid g\in Sp(1)\}\right)=&\left\{\psi(g\cdot v)\mid g\in Sp(1)\right\}\\
=&\left\{r^{-1}\phi(g\cdot v)\mid g\in Sp(1)\right\}\\
=&\left\{r^{-1}g\phi(v)\mid g\in Sp(1)\right\}\\
=&\left\{h\phi(v)\mid h\in Sp(1)\right\}\\
\end{align*}
 In particular, $\psi:\S^{4n+3}\to \S^{4n+3}$ is a foliated isometry, inducing a map $\tilde{\psi}:B\to\HP^n$ which commutes with the projections.
\end{enumerate}

To summarize: there exists a foliated isometry $\psi:(\S^{4m+3},\bar{\mathcal{F}})\to (\S^{4m+3},\tilde{\mathcal{F}})$ (where $\tilde{\mathcal{F}}$ is the foliation by the fibers of the standard quaternionic Hopf action), which is in addition $U(1)$-equivariant with respect to the standard complex Hopf action on both sides. This induces isometries $\bar{\psi}:\CP^{2m+1}\to \CP^{2m+1}$ and $\tilde{\psi}:B\to \HP^m$ which fit in the following commutative diagram:

\begin{center}
\begin{tikzcd}
\S^{4m+3}\arrow[d, "\psi"]\arrow[r, "\bar H_{2m+1}"]& \CP^{2m+1}\arrow[d,"\bar{\psi}"] \arrow[r]& B\arrow[d,"\tilde\psi"]\\
\S^{4m+3}\arrow[r, "H_{2m+1}"]& \CP^{2m+1} \arrow[r,"T"]& \HP^{m}
\end{tikzcd}
\end{center}

In particular, the Riemannian submersion $\CP^{2m+1}\to B$ is congruent to the standard twistor fibration $T$.

\section{$M$ \hili{homotopic} to $\HPn$}

For $\HPn$, consider a \hili{smooth homotopy equivalence} $f:M\to\HP^{n}$ and the quaternionic Hopf fibration $H_{n}:\S^{4n+3}\to\HP^{n}$. We can once again construct the following ``Hopf-like'' principal $Sp(1)$-bundle:

\begin{center}
\begin{tikzcd}
f^*(\S^{4n+3})\arrow[r,"F"]\arrow[d,"\bar{H}_n"]& \S^{4n+3}\arrow[d,"{H}_n"]\\
M\arrow[r,"f"]&\HPn
\end{tikzcd}
\end{center}

Again $F:f^{*}(\S^{4n+3})\to\S^{4n+3}$ is a \hili{homotopy equivalence}, and we can endow $f^{*}(\S^{4n+3})$ with a metric that makes $\bar{H}_{n}$ a Riemannian submersion and pullback $\fol$ through $\bar{H}_{n}$ to get the Riemannian foliation $(f^{*}(\S^{4n+3}),\bar{H}_{n}^{-1}(\fol)=:\bar{\fol})$.\\

With the same notation and arguments as before, 
\[
\dim(\bar{\mathcal{L}_{q}})=\dim(\mathcal{L}_{p})+\dim((\bar{H}_{n})^{-1}(p))=\dim(\mathcal{L}_{p})+\dim(Sp(1))=k+3>3\text{.}
\]

By Theorem \ref{thmfoliationsonspheres}, $4n+3=15$ so that $n=3$ and $\bar{\mathcal{L}_{q}}$ is homeomorphic to $\S^7$, meaning $\dim(\mathcal{L}_{p})=4$.

Consider the twistor bundle $T:\CP^{7}\to\HP^{3}$. This is a Riemannian submersion and a fibration with fibres homeomorphic to $Sp(1)/\S^1\simeq \S^{2}$, so that we can proceed as before to get the following:

\begin{center}
\begin{tikzcd}
f^*(\CP^7)\arrow[r,"\tilde{F}"]\arrow[d,"\bar{T}"]& \CP^7\arrow[d,"T"]\\
M\arrow[r,"f"]&\HP^3
\end{tikzcd}
\end{center}

Once again $\tilde{F}:f^{*}(\CP^{7})\to\CP^{7}$ is a \hili{homotopy equivalence} and we can endow $f^{*}(\CP^{7})$ with a metric that makes $\bar{T}$ a Riemannian submersion. We can thus pullback $\fol$ through $\bar{T}$ to get the Riemannian foliation $$(f^{*}(\CP^{7}),(\bar{T})^{-1}(\fol)=:\tilde{\fol}),$$ whose leaves have dimension $\dim \mathcal{L}+\dim T^{-1}(p)=6$. However, by the previous section there are no 6-dimensional foliations on manifolds \hili{homotopic} to $\CP^{2n+1}$. Therefore, no non-trivial Riemannian foliations can occur on spaces \hili{homotopic} to $\HP^{n}$.

\section{$M$ \hili{homotopic} to $\OPP$}

We are left with the easiest case, in which $M$ is \hili{homotopic} to $\OPP$. Recall that $H^{*}(M;\Z_{2})\cong\Z_{2}[c]/(c^{3})$ with $c$ generator of $H^{8}(M;\Z_{2})$. Furthermore, since the total Stiefel-Whitney class $w(TM)$ is a homotopy invariant, we have $w(TM)=w(T\OP^2)=1+c+c^{2}$ (see \cite{milnor}, p. 134).

If a (not necessarily Riemannian) foliation were to exist on $M$, just as in Section \ref{SS:case2} we would have $TM\cong\mathcal{H}\oplus\mathcal{V}$ and $w(TM)=w(\mathcal{V})\cup w(\mathcal{H})$. Since $\mathcal{H}$ and $\mathcal{V}$ have rank strictly lower than $16$ (as otherwise the leaves would either be $0$ or $16$-dimensional and the foliation would be trivial), their total Stiefel-Whitney classes must be of the form $w(\mathcal{V})=1+ac$, $w(\mathcal{H})=1+bc$ for some $a,b\in\Z_{2}$, resulting in $$1+c+c^{2}=(1+ac)(1+bc)=1+(a+b)c+ab c^2\qquad\textrm{in }\Z_{2}[c]/(c^{3})$$
which is impossible. Thus, $M$ does not admit any non-trivial foliation.

\section{\hili{Foliations on exotic $\CP^n$'s}}
In this section we prove Theorem \ref{MT:exotic}.

Let $(M, \F)$ be a Riemannian foliation on a manifold $M$ homotopy equivalent to $\CP^{2m+1}$. By Theorem \ref{thmfoliationsonCROSSes}, $\F$ is simple and given by the fibers of a Riemannian submersion $M\to B$ obtained by pulling back the twistor fibration $\CP^{2m+1}\to \HP^m$ via a homotopy equivalence $\varphi:B\to \HP^m$:

\begin{center}
\begin{tikzcd}
M\arrow[r,"f"]\arrow[d,"\pi"]& \CP^{2m+1}\arrow[d,"T"]\\
B\arrow[r,"\varphi"]&\HP^m
\end{tikzcd}
\end{center}
Notice that $\pi^*:H^*(B;\Z)\to H^*(M;\Z)$ is injective, and $f^*:H^*(\CP^{2m+1};\Z)\to H^*(M;\Z)$ is an isomorphism. In particular, we can use $f^*$ to identify $H^*(\CP^{2m+1};\Z)$ with $H^*(M;\Z)$, and $\pi^*$ to identify $H^*(B;\Z)$ with its image in $H^*(M;\Z)$.

Let $\mathcal{V}\to \CP^{2m+1}$ denote the vertical bundle $\ker T$. Then, using the identifications above, one has the following:
\[
TM\simeq f^*(\mathcal{V})\oplus \pi^*TB\qquad \Rightarrow\qquad p(TM)=p(TB)\cup p(\mathcal{V})
\]
Since $p(\mathcal{V})\in H^*(M;\Z)$ is invertible and independent of $M$ and $B$, we see that $p(TM)$ and $p(TB)$ uniquely determine each other. Furthermore, from $\varphi=id$ and $\pi=T$ it follows that $p(TM)=p(T\CP^{2m+1})$ if and only if $p(TB)=p(T\HP^m)$.

If $m=1$ then $B$ is 4-dimensional and homotopy equivalent to $\HP^1$. By e.g. Hirzebruch's Signature Theorem, we have $p(TB)=p(T\HP^1)$ hence, by the discussion above, $p(TM)=p(T\CP^3)$. Since $M$ is homotopy equivalent to $\CP^3$, in particular it is simply connected, with $w_2(TM)=w_2(T\CP^3)=0$. By Wall's classification of closed simply connected spin 6-manifolds \cite[Theorem 5]{Wal66}, $M$ and $\CP^3$ are diffeomorphic because they have the same invariants, namely $H^2(M,\Z)\simeq \Z$, $H_3(M,\Z)=0$, Pontryagin class $p_1(TM)$, and cubic intersection form $\mu:H^2(M;\Z)^3\to \Z$, $\mu(x,y,z)=[M]\cap(x\cup y\cup z)$ (in this case, $\mu(a,b,c)=abc$).

Suppose now that $m\geq 2$. By \cite[Theorem 2]{Hsi66}, there are infinitely many manifolds $B_i$ homotopy equivalent to $\HP^m$ but with mutually distinct rational Pontryagin classes $p(TB_i)$. Letting $\varphi_i:B_i\to \HP^m$ the  homotopy equivalences and $M_i$ the corresponding pullbacks of $T$ through $\varphi_i$, we thus obtain that $p(M_i)$ are also mutually distinct.




\printbibliography[heading=bibintoc, title={Bibliography}]

@article{lytchakwilking,
   title={Riemannian foliations of spheres},
   volume={20},
   ISSN={1465-3060},
   DOI={10.2140/gt.2016.20.1257},
   number={3},
   journal={Geometry \& Topology},
   publisher={Mathematical Sciences Publishers},
   author={Lytchak, Alexander and Wilking, Burkhard},
   year={2016},
   month=jul, 
   pages={1257–1274}
}

@book{milnor,
  title={Characteristic Classes},
  author={Milnor, J.W. and Stasheff, J.D.},
  isbn={9780691081229},
  lccn={lc72004050},
  series={Annals of mathematics studies},
  year={1974},
  publisher={Princeton University Press}
}

@article {gromollgrove,
    AUTHOR = {Gromoll, Detlef and Grove, Karsten},
     TITLE = {The low-dimensional metric foliations of {E}uclidean spheres},
   JOURNAL = {J. Differential Geom.},
  FJOURNAL = {Journal of Differential Geometry},
    VOLUME = {28},
      YEAR = {1988},
    NUMBER = {1},
     PAGES = {143--156},
      ISSN = {0022-040X,1945-743X},
   MRCLASS = {53C12 (57R30)},
  MRNUMBER = {950559},
MRREVIEWER = {Richard\ H.\ Escobales, Jr.}
}

@article{browder,
    ISSN = {00029947},
    author = {Browder, William},
    journal = {Transactions of the American Mathematical Society},
    number = {2},
    pages = {353--375},
    publisher = {American Mathematical Society},
    title = {Higher Torsion in H-Spaces},
    volume = {108},
    year = {1963}
}

@article{Escobales,
    ISSN = {00029939, 10886826},
    author = {Escobales, Richard H.},
    journal = {Proceedings of the American Mathematical Society},
    number = {2},
    pages = {280--284},
    publisher = {American Mathematical Society},
    title = {Sufficient Conditions for a Bundle-Like Foliation to Admit a Riemannian Submersion onto its Leaf Space},
    volume = {84},
    year = {1982}
}

@article{Hermann,
 ISSN = {0003486X, 19398980},
 author = {Hermann, Robert},
 journal = {Annals of Mathematics},
 number = {3},
 pages = {445--457},
 publisher = {[Annals of Mathematics, Trustees of Princeton University on Behalf of the Annals of Mathematics, Mathematics Department, Princeton University]},
 title = {On the Differential Geometry of Foliations},
 volume = {72},
 year = {1960}
}

@article {Wil01,
    AUTHOR = {Wilking, Burkhard},
     TITLE = {Index parity of closed geodesics and rigidity of {H}opf
              fibrations},
   JOURNAL = {Invent. Math.},
  FJOURNAL = {Inventiones Mathematicae},
    VOLUME = {144},
      YEAR = {2001},
    NUMBER = {2},
     PAGES = {281--295},
      ISSN = {0020-9910,1432-1297},
   MRCLASS = {53C12 (53C22 53C24)},
  MRNUMBER = {1826371},
       DOI = {10.1007/PL00005801},
}

@article {Ran85,
    AUTHOR = {Ranjan, Akhil},
     TITLE = {Riemannian submersions of spheres with totally geodesic
              fibres},
   JOURNAL = {Osaka J. Math.},
  FJOURNAL = {Osaka Journal of Mathematics},
    VOLUME = {22},
      YEAR = {1985},
    NUMBER = {2},
     PAGES = {243--260},
      ISSN = {0030-6126},
   MRCLASS = {53C20},
  MRNUMBER = {800969},
MRREVIEWER = {Richard\ H.\ Escobales, Jr.},
}

@article {Des02,
    AUTHOR = {Dessai, Anand},
     TITLE = {Homotopy complex projective spaces with {$Pin(2)$}-action},
   JOURNAL = {Topology Appl.},
  FJOURNAL = {Topology and its Applications},
    VOLUME = {122},
      YEAR = {2002},
    NUMBER = {3},
     PAGES = {487--499},
      ISSN = {0166-8641,1879-3207},
   MRCLASS = {58J26 (11F23 19L47 57R20 57S25)},
  MRNUMBER = {1911696},
MRREVIEWER = {Xiaonan\ Ma},
       DOI = {10.1016/S0166-8641(01)00192-4},
       url = {https://www.sciencedirect.com/science/article/pii/S0166864101001924}
}

@book {BT82,
    AUTHOR = {Bott, Raoul and Tu, Loring W.},
     TITLE = {Differential forms in algebraic topology},
    SERIES = {Graduate Texts in Mathematics},
    VOLUME = {82},
 PUBLISHER = {Springer-Verlag, New York-Berlin},
      YEAR = {1982},
     PAGES = {xiv+331},
      ISBN = {0-387-90613-4},
   MRCLASS = {57R19 (55-02 58-01 58A12)},
  MRNUMBER = {658304},
MRREVIEWER = {Hansklaus\ Rummler},
}

@article{Sma,
 ISSN = {00029939, 10886826},
 URL = {http://www.jstor.org/stable/2033664},
 author = {Stephen Smale},
 journal = {Proceedings of the American Mathematical Society},
 number = {4},
 pages = {621--626},
 publisher = {American Mathematical Society},
 title = {Diffeomorphisms of the 2-Sphere},
 volume = {10},
 year = {1959}
}

@article{Hsi66,
    AUTHOR = {Hsiang, Wu-chung},
     TITLE = {A note on free differentiable actions of {$S\sp{1}$} and
              {$S\sp{3}$} on homotopy spheres},
   JOURNAL = {Ann. of Math. (2)},
  FJOURNAL = {Annals of Mathematics. Second Series},
    VOLUME = {83},
      YEAR = {1966},
     PAGES = {266--272},
      ISSN = {0003-486X},
   MRCLASS = {57.47},
  MRNUMBER = {192506},
MRREVIEWER = {F.\ Hirzebruch},
       DOI = {10.2307/1970431},
       URL = {https://doi.org/10.2307/1970431},
}

@article{Wal66,
	title = {Classification problems in differential topology. {V}. {On} certain 6-manifolds},
	volume = {1},
	issn = {0020-9910,1432-1297},
	url = {https://doi.org/10.1007/BF01425407},
	doi = {10.1007/BF01425407},
	journal = {Inventiones Mathematicae},
	author = {Wall, C. T. C.},
	year = {1966},
	mrnumber = {215313},
	pages = {355--374; corrigendum, ibid. 2 (1966), 306},
}

\end{document}